\documentclass[11pt]{article}
\usepackage{amsmath, amsthm}
\usepackage{amssymb}
\usepackage[hypertex]{hyperref}
\usepackage{graphicx}

\theoremstyle{theorem}
\newtheorem{theorem}{Theorem}
\newtheorem{proposition}[theorem]{Proposition}
\newtheorem{lemma}[theorem]{Lemma}

\theoremstyle{definition}
\newtheorem{definition}[theorem]{Definition}
\newtheorem{example}[theorem]{Example}
\newtheorem{problem}[theorem]{Problem}

\theoremstyle{remark}
\newtheorem{remark}[theorem]{Remark}

\theoremstyle{notation}

\newcommand{\al}{\alpha}
\newcommand{\be}{\beta}

\newcommand {\CC}{\mathbb{C}}

\newcommand {\PP}{\mathbb{P}}

\newcommand {\G}{\mathfrak{G}}

\newcommand{\Ann}{{\rm Ann}}

\newcommand{\gb}{\mathfrak{b}}

\newcommand{\Hom}{{\rm Hom}}

\renewcommand{\deg}{{\rm deg}}

\title{A characterization of the Macaulay dual generators for quadratic complete intersections} 
\author{T.\  Harima, Niigata University \\ Department of Mathematics Education, Niigata, 950-2181 Japan
\thanks{Supported by JSPS KAKENHI Grant 15K04812.} 
\and A.\ Wachi, Hokkaido University of Education \\  Department of Mathematics,  
\\ Kushiro, 085-8580 Japan
\and  J.\ Watanabe, Tokai University \\ Department of Mathematics, Hiratsuka, 259-1292 Japan
}
\begin{document}

\maketitle
\date{}

\def\pa{{\partial}}
   
\begin{abstract} 
  Let $F$  be a homogeneous polynomial in $n$ variables of degree $d$ over a field $K$.
  Let $A(F)$ be the associated Artinian graded $K$-algebra.
  If $B \subset A(F)$  is a subalgebra of $A(F)$ which is Gorenstein with the
  same socle degree as $A(F)$, we describe the Macaulay
  dual generator for $B$ in terms of  $F$.
  Furthermore when $n=d$, we give necessary and sufficient
  conditions on the polynomial $F$ for $A(F)$ to be a complete intersection.  
\end{abstract}

\section{Introduction}
Let $R=K[x_1, \ldots, x_n]$ be the polynomial ring in $n$ variables over a field of characteristic zero and
$R_d$ the homogeneous space of degree $d$.  For $F \in R_d$,
let $A=A(F)$ be the graded Artinian Gorenstein algebra associated with $F$. So $A$ has socle degree $d$ and embedding dimension at most $n$.
It is a long standing problem to characterize forms $F \in R_d$ for which  the associated Artinian Gorenstein algebras $A(F)$
are complete intersections.
If $F$ is a monomial, then $A(F)$ is a monomial complete intersection.  The only other known cases are a few sporadic 
examples (cf.\cite{solomon_1}, \cite[Examples 2.82--2.85]{HMMNWW}) that occur as the algebra of co-invariants by pseudo reflection groups.
It seems that there is a tendency among the experts to think there are no easily verifiable conditions which enable us to tell,
for a given $F$, whether or not the algebra  $A(F)$ is a complete intersection.

However,
it is easy to see that if the degree of $F$ is less than $n$,
then $A(F)$ cannot be a complete intersection, since the socle degree of $A(F)$ is
equals to the degree of the Jacobian of the generators.  
When $A(F)=R/I$ is a graded complete intersection with quadratic generators for the defining ideal $I=\Ann _R(F)$, then
the degree of $F$ is $n$. In Theorem~\ref{main_theorem} of this paper 
we give necessary and sufficient conditions on a form $F$ for $A(F)=R/I$ to be a
complete intersection.

There is yet another result of this paper.  We discuss the relation of  Macaulay dual generators for
two Artinian Gorenstein algebras $A$ and $B$, with $A \supset B$, when the two algebras have the same socle degree.
This was one of the topics discussed in the workshop at BIRS in March 2016, under the title ``The Lefschetz Properties and Artinian algebras.''
There is a good reason to think that many complete intersections can be obtained as subrings of
quadratic complete intersections (cf.\ \cite{harima_wachi_watanabe_1} \cite{mcdaniel}).
We will show that a Macaulay dual generator for $B$ can be obtained from that of
$A$ by substituting the linear forms for
the variables with duplications allowed. In Theorem~\ref{main_thm_2} we show that the Gorenstein
Artinian algebra $A$ has a sub-quotient $B$ of the same socle degree if and only if a Macaulay dual generator for
$B$ can be obtained from that of $A$ by substituting the linear forms for the variables.
This is independent of Theorem~\ref{main_theorem} and in this theorem
the socle degree and embedding dimension of $A$ are arbitrary. 

When we speak about the Macaulay dual generator of a Gorenstein algebra, it is important to specify the structure of
the inverse system or in the modern term, the injective hull of the residue field. The injective hull does not have a structure
of a ring; it is possible, however, to regard it as the divided power algebra,
induced by the natural structure of the Hopf algebra associated with the polynomial ring.
In characteristic zero, the divided power algebra is the same as the polynomial ring. Throughout \S\S 2--3,
we assume, for simplicity,  that the characteristic of the ground field is zero, and assume that the injective
hull is the same as the polynomial ring itself but the action of the algebra is defined through differentiation.
Nonetheless all arguments are valid for a positive characteristic $p$ provided that $p$ is greater than the degree of $F$.
Verification is left to the reader. A basic fact on the Macaulay's double annihilator theorem is summarized in the
Appendix based on the treatment in Meyer-Smith~\cite{smith_1002}.  The reader may also wish to consult
Geramita~\cite{Ge} and Iarrobino-Kanev~\cite{IK} for the treatment of the inverse system.    

The authors would like to thank Tony Iarrobino very much for many comments to improve this paper. 
The third author thanks Larry Smith very much for the suggestion of the usage of the divided power
algebra for the inverse system of Macaulay.

\section{Some necessary conditions for a homogeneous form to define a quadratic complete intersection} 

Throughout this section  $K$ denotes a field of characteristic $0$, and 
$R=K[x_1, x_2, \ldots, x_n]$ denotes the polynomial ring over $K$.  We assume each variable has degree $1$. 
We denote by $R_d$ the homogeneous space of $R$ of degree $d$.  
Thus we may write 
\[R=\bigoplus _{d=0}^{\infty} R _d.\]

We  regard $R$ as an $R$-module via the operation ``$\circ$'' defined by  
\[
f(x_1, x_2, \ldots, x_n)\circ F=f\left(\frac{\pa}{\pa x_1},\frac{\pa}{\pa x_2}, \ldots,
\frac{\pa}{\pa x_n}\right)F,\] 
for $(f, F) \in R \times R$. 
With this operation $R$ is the injective hull of the residue field in the category of 
finitely generated modules (see Appendix).  
Thus if  $f \in R_i$  and  $F \in R_d$,  then  $f\circ F$  is an element of $R_{d-i}$.  
For $F \in R$, $\Ann _R(F)$ denotes 
\[\Ann _R(F)= \{ f \in R| f\circ F=0 \}.\]
It is  the {\bf annihilator} of $F$. 
$A(F)$ denotes the  algebra $A(F)=R/\Ann _R(F)$. 
We will say that $A(F)$ is the {\bf Gorenstein algebra defined by} $F$ or simply  $A(F)$ {\bf is defined by} $F$.  
We will call the vector space $\Ann(F) _2 \subset R_2$ the {\bf quadratic space defined by  $F$} and 
denote it by $Q(F)$.
Namely, the quadratic space  $Q(F)=\Ann(F) _2 \subset R_2$ is the kernel of 
the homomorphism
\[f \in R_2 \mapsto f \circ F \in R_{d-2}.\] 
Note that 
we have the exact sequence 
\[0 \to Q(F) \to R_2 \to  R_2/ \Ann _R(F)_2 \to 0.\]

For a graded vector space $V=\bigoplus _{i=0}^{\infty} V_i$, we write 
\[H_V (T)=\sum _{i=0}^{\infty}(\dim _K V_i) T^i\]
for the Hilbert series of $V$. 

\begin{definition}   \label{definition_qci} 
An Artinian algebra $A=R/I$ will be called a {\bf quadratic complete intersection} if it is a complete intersection  
and the Hilbert series is $(1+T)^d$ for some $d$.   (The ideal $I$ may contain linear forms.) 
\end{definition}

\begin{proposition} \label{necessary_condition} 
Let $F \in R=K[x_1, x_2, \ldots, x_n]$ be a homogeneous polynomial.  
Suppose that $H_{A(F)}(T) = (1+T)^n.$ 
Then we have 
\begin{enumerate}
\item[{\rm (1)}] $\deg \;F=n$. 
\item[{\rm (2)}] 
No linear forms are contained in $\Ann _R(F)$. 
\item[{\rm (3)}]
The partial derivatives $\frac{\pa F}{\pa x_1}, \frac{\pa F}{\pa x_2}, \ldots,  \frac{\pa F}{\pa x_n}$ 
are linearly independent. 
\item[{\rm (4)}]
The quadratic space $Q(F)$ defined by $F$ has dimension  $n$. 

\item[{\rm (5)}] 
$\dim _K (R_2\circ F) =\dim _K(R_{n-2}\circ F) ={n \choose 2}$.

\end{enumerate}
\end{proposition}

\begin{proof} 
\begin{enumerate}
\item[(1)] 
Recall that the homomorphism of $R$-modules 
\[R \to R \]
defined by $f \mapsto f\circ F$ induces the degree reversing isomorphism of vector spaces 
\[R/\Ann _R(F)=A(F) \to R\circ F \subset R,\]
\[\overline{f} \mapsto f\circ F,\]
where 
\[ {\overline f} = f \bmod \Ann _R(F).\]
This shows that if the algebra $A(F)$ has the Hilbert series $(1+T)^n$, then $F$ has degree $n$.

\item[(2)] 
Note that $A(F)_1=R_1/\Ann _R(F)_1$.  Since $\dim _K A(F)_1=n$, this shows that 
$\Ann _R(F)_1=0$.  Hence $\Ann _R(F)$ contains no linear forms.

\item[(3)]
Note that $A(F)_1 \cong R_{1}\circ F$.  Since $\dim _K R_1\circ F = n$, this shows that 
the first partials of $F$ are linearly independent. 

\item[(4)] (5)
Consider the exact sequence 
\[0 \to Q(F) \to R_2 \to R_2/\Ann _R(F)_2 \to 0.\]
Note that $\dim R_2= {n+1 \choose 2}$  
and $\dim _K A(F)_2 = {n \choose 2}$.  The assertions follow from  
the isomorphisms  
$A(F)_2 =  R_2/\Ann _R(F)_2 
\cong   
R_2 \circ F     \cong  
R_{n-2} \circ F$. 
\end{enumerate}
\end{proof}

\section{A characterization of the Macaulay dual generator for quadratic complete intersections}

Following is a characterization of $F\in R$ which defines a quadratic complete intersection.    
\begin{theorem}  \label{main_theorem}
As before $R$ denotes the polynomial ring in $n$ variables over a field $K$ of characteristic zero.  
Let $F \in R$ be a polynomial of degree $n$. 
Suppose that the partial derivatives $\frac{\pa F}{\pa x_1}, \cdots, \frac{\pa F}{\pa x_n}$ are linearly independent. 
Then the Artinian Gorenstein algebra $A(F)=R/\Ann _R(F)$  
is a quadratic complete intersection if and only if one of the following conditions is satisfied. 

\begin{enumerate}
\item[{\rm (1)}] 
The quadratic space  $Q(F)$ is $n$-dimensional and generates $\Ann _R(F)$ as an ideal of $R$.  
\item[{\rm (2)}]
The quadratic space  $Q(F)$ contains a  regular sequence of length $n$ in $R$.  
\end{enumerate}
\end{theorem}

\begin{proof}
Assume that $A(F)$ is a quadratic complete intersection.  
Then $\Ann _R(F)$ is  generated by a regular sequence consisting of $n$ homogeneous polynomials of degree two. 
Hence we have both (1) and (2).

Conversely assume (2). Let $I$ be the ideal generated by a regular sequence in $Q(F)$. 
Then we have a surjective map 
\[ R/I \to R/\Ann _R(F)=A(F).  \] 
Since $R/I$ and $R/\Ann _R(F)$ are Gorenstein with the same socle degree, we have  $A(F)=R/I$.  
(To see this recall that an Artinian
Gorenstein local ring has the smallest nonzero ideal.)  
Assume (1). Then a basis for $Q(F)$ is a regular sequence in $R$. Hence $\Ann _R(F)$ is generated by a 
regular sequence consisting of quadrics.    
\end{proof}

\begin{remark}
Theorem~\ref{main_theorem} gives us an algorithm which determines whether or not the algebra 
$A(F)=R/\Ann_R(F)$ is a quadratic  complete intersection for  
a given $F \in K[x_1, \ldots, x_n]$.   
The algorithm proceeds as follows.
\begin{enumerate}
\item[(1)]  
Let  $F \in K[x_1, \ldots, x_n]_n$ be a homogeneous form of degree  $n$.  
\item[(2)]    
Check if the partials $\frac{\pa F}{\pa x_1}, \ldots, \frac{\pa F}{\pa x_n}$ are linearly independent.   
If they are linearly dependent, $F$ reduces to a polynomial of a smaller 
number of variables, but it has  degree  $n$. So in this case $A(F)$ is not a quadratic complete intersection. 
(It could be a complete intersection with a smaller embedding dimension than $n$.) 
If they are linearly independent, compute the second  partials of $F$. 
If $\dim _K(R_2\circ F) \neq {n \choose 2}$ or equivalently 
$\dim _K Q(F) \neq n$, then $A(F)$ cannot be a quadratic complete intersection. 
\item[(3)] 
If $\dim _K Q(F) = n$, let $\langle f_1, \ldots, f_n \rangle$ be a $K$-basis of $Q(F)$. 
Compute the rank of the vector space $V=R_{n-1}f_1 + R_{n-1}f_2 + \cdots + R_{n-1}f_n$. 
If $\dim _K V={2n \choose n+1 }$, then $\Ann _R(F)$ is a quadratic complete intersection; otherwise it is not.
(Note that $V \subset R_{n+1}$ and  $\dim _K R_{n+1}={2n \choose n+1}$,  and  
 $V=R_{n+1}$ if and only if $f_1, f_2, \ldots, f_n$ is 
a complete intersection.)
\end{enumerate}
\end{remark}

\begin{remark}  
Suppose that $\dim Q(F)=n$ and $Q(F)=\langle f_1, f_2, \ldots , f_n   \rangle$.
It is easy to see that the following conditions are equivalent.
\begin{enumerate}
\item[(1)]
$f_1, f_2, \ldots, f_n$ is a regular sequence.

\item[(2)]
$R_{n-1}f_1 +  R_{n-1}f_2 +  \cdots + R_{n-1} f_n = R_{n+1}$. 

\item[(3)]

The initial ideal of  $(f_1, f_2, \ldots, f_n)$   contains all high powers of the variables.

\item[(4)]
The resultant of $f_1, f_2, \ldots, f_n$ does not vanish. (For the theory of resultants see \cite{GKZ}. 
There is a related result in  \cite{harima_wachi_watanabe_2}.)   
\end{enumerate}
\end{remark}

\begin{remark}  
  Tony Iarrobino pointed that Thoerem~\ref{main_theorem} can be generalized to any homogeneous polynomial
  $F\in R_{n(d-1)}$ to define a complete intersection with generators of any uniform degree $d$. 
  We confined ourselves to the quadratic case ($d=2$),  since the generalization is straightforward, and since we had in mind  the results of 
  \cite{harima_wachi_watanabe_1} and \cite{harima_wachi_watanabe_2}. 
  \end{remark}

\begin{example}   
Let $R=K[x,y,z]$. Consider $F \in R_3$.  If the partials $F_x, F_y, F_z$ are linearly independent, 
then the Hilbert series for $A(F)$ is $(1+T)^3$, in which case $\dim _KQ(F) =3$. 
So $A(F)$ is a quadratic complete intersection for most of $F \in R_3$.  Following are exceptional cases.       
These examples are due to Buczy\'{n}ska et al\@. \cite{teitler}. 
\begin{enumerate}
\item[(1)]
 $F=-x^3+y^2z$.  
It is easy to see that  $Q(F) \supset \langle z^2, xz \rangle$, and that these are not a regular sequence.   
So we may conclude $Q(F)$ cannot be a complete intersection.
The fact is that $\Ann _R(F)$ is a 5 generated Gorenstein ideal: 
\[\Ann _R(F)=(z^2, xz,xy, y^3, x^3+3y^2z).\]

\item[(2)]
  $F=x^2y+y^2z$.  We can apply the same argument as above to see that $\Ann _R(F)$ is not a complete intersection.
\[\Ann _R(F)=(z^2, xz, x^2-yz, y^3, xy^2).\] 
\end{enumerate}
\bigskip

The classification of ternary cubics is known over the complex number field.  
The  complete sets of orbits in the parameter space $\CC \PP ^9$ for the ternary cubics by the general linear group ${\rm GL}(3, \CC)$ 
can be found in \cite{teitler} Section 2.  
\end{example}


\begin{example}    

Let $R=K[w,x,y,z]$. 
\begin{enumerate}

\item[(1)]
Consider $F=(w-x)(y-z)(w^2+x^2+y^2+z^2)$.  Then we have:
\[Q(F)=\langle wx-yz,  y^2+ 4yz + z^2, w^2 + 4yz + x^2, (w+x)(y+z)  \rangle\]
With an aid of a computer algebra system, it is easy to see that the ideal 
generated by these elements is a complete intersection.  Hence $A(F)$ is a quadratic complete intersection.

\item[(2)]
Consider $F=(w-x)(y-z)(w+x+y+z)^2$.  Then we have:
\[F_w+F_x-F_y-F_z=0.\]
This shows that $\Ann _R(F)$ contains a linear form. $A(F)$ is not a quadratic complete intersection 
but is a complete intersection with embedding dimension three.  
In fact 
\[\Ann _R(F)=(w+x-y-z, x(x-y-z)+yz, (y+z)^2, z^2(3y-z). )\]

\item[(3)]
 $F=(wx)^2-(yz)^2$. 
It is easy to  see that $Q(F)=\langle  wy,wz, xy,xz   \rangle$. 
So this is not a quadratic complete intersection. 
It happens that the Hilbert series is $(1+T)^4$. 
However, 
\[\Ann _R(F)=(wy,wz, xy,xz, w^3, x^3, y^3, z^3, (wx)^2+(yz)^2     ). \] 
\end{enumerate}
\end{example}


\begin{example}  
Let $R=K[v,w,x,y,z]$.
\begin{enumerate}

\item[(1)]
$F=vwxyz+wxyz^2$. It is easy to see that $\Ann _R(F)$ contains 5 quadratic relations. 
In fact $(v^2, w^2, x^2, y^2, z^2- 2vz ) \subset \Ann _R(F)$.
So $\Ann _R(F)$ is a complete intersection. 
\item[(2)]
$F=vwxyz+xyz^3$. $\Ann _R(F)=(v^2, w^2, x^2, y^2,z^2- 6vw )$.  
Similarly to the previous example, this is a complete intersection.  
\item[(3)]
$F=vwxyz+yz^4$. 
It is easy to see that we have the relations 
\[\left(\frac{\pa}{\pa v}\right)^2F=\left(\frac{\pa}{\pa w}\right)^2F=\left(\frac{\pa}{\pa x}\right)^2 F
=\left(\frac{\pa}{\pa y}\right)^2 F=0.\]

With a little contemplation we see that no more quadratic relations are possible. 
So this is not a complete intersection. 
In fact we can compute 
$\Ann _R(F)=(v^2, w^2, x^2, y^2, wz^2, vz^2, xz^2, z^3- 24vwx)$. 
\end{enumerate}
\end{example}

\begin{problem}
For what binomial $F \in K[x_1, x_2, \ldots, x_n]_n$ is $\Ann _R(F)$ a complete intersection? 
(Define $F$ to be a binomial by $F=\al M + \be N$, where $M, N$ are
power products of variables and $\al, \be \in K$.)  
\end{problem}



\begin{remark}
  The vector  space $R_n$ may be regarded as the  parameter
  variety for the Gorenstein algebras of socle degree $n$ with embedding dimension at most $n$.
  Thus the projective space $\PP ^N$   where  $N={2n-1 \choose n}-1$
  is the  paremeter space for such Gorenstein algebras.  By the Double Annihilator Theorem of Macaulay, 
  each orbit of the general linear group $GL(n)$ contains precisely one isomorphism type of Gorenstein algebras.
  (See \cite{smith_1002}.)
  Thus the dimension of the parameter space for the isomorphism types of Gorenstein algebras with
  socle degree  $n$  and    embedding dimension at most $n$ is
  \[\left({2n-1 \choose n}-1\right)- (n^2-1).\]
  (We roughly estimated that each orbit is $(n^2-1)$-dimensional.) 

  On the other hand, the set of quadratic complete intersections may be parametrized by the $n$-dimensional subspaces
  in $R_2$. Thus the  dimension of the parameter space is $n \times {n \choose 2}$.
  The linear transformation of the variables gives us an isomorphism of such complete intersections.
  Hence the dimension of the parameter space for the isomorphism types is
  \[n\times {n \choose 2} - (n^2-1).\]

  Here is a list of these dimensions for small values of $n$.

  \medskip

  \begin{center}
  $
  \begin{array}{c|ccccc} \hline 
    n                       & 2   & 3     & 4  &  5  &  6     \\ \hline 
    {2n-1 \choose n} - n^2   & -   & 1     & 19 & 101  &  426  \\ \hline
    n{n \choose 2} - n^2+1   & -   & 1     & 9  & 26   &  55   \\ \hline 
    \end{array}
  $

\end{center}
\end{remark}

\medskip
\medskip
\begin{remark}
  Suppose $F \in R_n$ defines a quadratic complete intersection. We may assume that the square free monomials
  are linearly independent in $A(F)$. (For this fact see \cite{harima_wachi_watanabe_2}.)  Let
  $B_k \subset R_k$ be the set of square free monomials of degree $k$ and define 
  the ${n \choose k} \times {n \choose k}$ matrix $H_k(F)$ as follows:
   \[H_k(F)= \left((\al\be) \circ F\right)_{(\al, \be) \in B_k \times B_k}.\]
   The rows and columns of  $H_k(F)$ are indexed by $B_k$.
   In \cite{maeno_watanabe} the authors call the determinant of  $H_k(F)$ the higher Hessian of order $k$ of $F$
   (with respect to the basis $B_k$). By \cite{watanabe_1} Theorem~4 the algebra
   $A(F)$ has the strong Lefschetz property if and only if
   \[\det H_k(F) \neq 0, \mbox{ for all } k=1,2, \ldots, [n/2].\]   
   We conjecture that  $\det H_k(F)$ does not vanish for $F \in R_n$ for all $k$,
   if $F$ defines a quadratic complete intersection.  
   This is a part of a larger conjecture which claims that all complete intersections over a field of characteristic zero
   have the strong Lefschetz property. For more detail see \cite{HMMNWW} Conjecture~3.46 and Theorem~3.76.  
\end{remark}

In the next example we show that there exists a Gorenstein algebra with the same Hilbert series
as a quadratic complete intersection, but fails the SLP. 
\begin{example}[R.\ Gondim]   
  Consider the polynomial 
  \[F=v^3wx+vw^3y+y^2z^3\]
  of degree  5  in 5 variables. 
  With an aid of a computer algebra system one sees that $A(F)$ has the Hilbert series  
  $(1+T)^5$, but $A(F)$ is not a complete intersection. 
  It is not difficult to see that the 2nd Hessian of $F$ with respect to certain bases  for $A_2$ and $A_3$ is identically zero,
  so the algebra $A(F)$ fails the SLP. 
  The set of square free monomials of degree 2 is linearly dependent in $A(F)_2$, but this is not essential to the failure of 
  the SLP. In fact if $F$ is expressed in generic variables, the sets of square-free monomials can be  bases for $A_2$ and $A_3$. 
This example was constructed by Gondim~\cite{gondim}. (See \cite{gondim}, Theorem~2.3 and the paragraph preceding it.)
\end{example}

\section{The Macaulay dual generator for a subring of a Gorenstein algebra}
In this section we consider Artinian algebras   $A$   over a field $K$ with the assumption 
characteristic $K$ is zero or greater than the socle degree of $A$. 
The socle degree and the embedding dimension of $A$ are arbitrary.  

\begin{theorem} \label{maeno's observation} 
Suppose that $A=\bigoplus _{i=0}^d A_i$ is a standard graded Artinian Gorenstein algebra with $A_d \neq 0$. 
Assume that $A_0=K$ is a field of characteristic $p =0 $ or $p > d$. 
Let  $\overline{x_1}, \overline{x_2}, \ldots, \overline{x_n} \in A_1$ be 
the images of the variables of the polynomial ring and  
let $\xi_1, \xi_2, \ldots, \xi _n \in K$. 
Fix a nonzero socle element $s \in A_d$. Define the map  
\[\Phi: K ^n \to K\]
which sends $(\xi_ 1, \xi _2, \ldots, \xi _n )$ to $c$, 
where $c$ is defined by 
\[(\xi _1\overline{x_1} + \xi _2\overline{x_2} + \cdots + \xi _n\overline{x_n})^d=cs.\]
Since $c$ is a function of $\xi _1, \xi _2, \ldots, \xi_n$, we may write 
$c=\Phi(\xi_1, \xi_2, \ldots, \xi _n)$. The map $\Phi$ is a polynomial map and 
$\Phi(\xi_1, \xi_2, \ldots, \xi _n)$,   
as a polynomial, is a Macaulay dual generator for the Gorenstein algebra $A$.     
\end{theorem}

\medskip  
This was proved in \cite{HMMNWW} Lemma~3.47. Here we give another proof.
\begin{proof} 
Let $R=K[x_1, x_2 , \ldots, x_n]$ be the polynomial ring 
and $F \in R_d$ a Macaulay dual generator for $A$.  
 Then  we have the isomorphism 
$R/ \Ann _R(F) \cong A$ defined by $x_i \mapsto \overline{x_i}$. 
Let $S=S(x_1, \ldots, x_n) \in R_d$ be a pre-image of a nonzero socle element $s \in A_d$.  
Put $\al= S\circ F \in K$. Since $S \not \in \Ann _R(F)$, we have $\al \neq 0$.  
Given $(\xi _1, \xi _2, \ldots, \xi _n) \in K^n$, 
we want to find $c=c(\xi _1, \ldots, \xi _n) \in K$ which satisfies 
\[cs = (\xi _1\overline{x_1} + \xi _2\overline{x_2} + \cdots + \xi _n\overline{x_n})^d.\]
Such $c$ should satisfy   
$(\xi _1 x_1 + \xi_2 x_2 + \cdots + \xi _n x_n)^d - cS \in \Ann _R(F)$.
Thus we should have 
\[((\xi _1 x_1 + \xi _2 x_2 + \cdots + \xi _n x_n)^d - cS )\circ F=0.\]
 We compute the left hand side as follows:
\[
\begin{array}{ll}
((\xi _1 x_1 + \xi _2 x_2 &  + \cdots + \xi _nx_n)^d - cS )\circ F   
 \\ \qquad &  =  (\xi _1 x_1 + \xi_2 x_2 + \cdots + \xi _n x_n)^d \circ F -cS \circ F     
 \\ \qquad &  = d!F(\xi _1, \ldots, \xi _n) - c \al 
\end{array}
\]
We used Lemma~\ref{watanabe_queen} which we prove below for the last equality.
It turned out that   
\[c=\frac{d!}{\al}F(\xi _1, \ldots, \xi _n).\]
Hence we have 
\[
\Phi(x_1,x_2,\ldots,x_n)=\frac{d!}{\al}F(x_1,x_2,\ldots,x_n). 
\]
\end{proof}


\begin{lemma}  \label{watanabe_queen}  
Let $R=K[x_1, x_2, \ldots, x_n]$ be the polynomial ring and  
let 
\[\xi_1, \xi_2, \ldots, \xi_n \in K\]
be any elements in $K$. Put 
\[D=\xi_1 \frac{\pa}{\pa x_1} + \xi_2 \frac{\pa}{\pa x_2} +  \cdots  + \xi_n \frac{\pa}{\pa x_n}.\]
Then for any homogeneous polynomial $F \in R_d$ of degree $d$, we have 
\[D^d F=d!F(\xi _1,  \xi _2, \ldots, \xi _n).\]
\end{lemma}


\begin{proof}
Let $F=\sum_{d_1+ \cdots +d_n=d} a_{(d_1,\ldots,d_n)} x_1^{d_1} \cdots x_n^{d_n}$. 
Then: 
\[
\begin{array}{rcl}
D^dF &=& (\xi_1 \frac{\partial}{\partial x_1} + \cdots + \xi_n \frac{\partial}{\partial x_n})^d F \\[1ex] 
&=& {\displaystyle \left(\sum_{d_1+ \cdots +d_n=d} \frac{d!}{d_1! \cdots d_n!}  
  \xi_1^{d_1}(\frac{\partial}{\partial x_1})^{d_1} \cdots \xi_n^{d_n}(\frac{\partial}{\partial x_n})^{d_n} \right) F} \\[1ex]
&=& 
{\displaystyle \sum_{d_1+\cdots+d_n=d} \frac{d!}{d_1! \cdots d_n!} 
\left(\xi_1^{d_1}(\frac{\partial}{\partial x_1})^{d_1} \cdots \xi_n^{d_n}(\frac{\partial}{\partial x_n})^{d_n} \right)
\left(a_{(d_1,\ldots,d_n)} x_1^{d_1} \cdots x_n^{d_n}\right) }   \\[1ex] 
&=&
{\displaystyle \sum_{d_1+\cdots+d_n=d} \frac{d!}{d_1! \cdots d_n!} a_{(d_1,\ldots,d_n)} 
\xi_1^{d_1} \cdots \xi_n^{d_n} d_1! \cdots d_n! }\\[1ex] 
&=& d!F(\xi_1,\ldots,\xi_n).
\end{array}
\]

\end{proof}

\medskip



\begin{definition}
Suppose that $F=F(x_1, \ldots, x_n)$ is a polynomial in $x_1, \ldots, x_n$ and  
$G=G(y_1, \ldots, y_m)$ is a polynomial in $y_1, \ldots, y_m$. (Assume that 
the sets $\{x_1, \ldots, x_n\}$ and $\{y_1, \ldots, y_m\}$ are independent sets of variables.)
We will say that $G$ is obtained from $F$ {\bf by substitution by linear forms}, if there exists a full rank 
$m \times n$ matrix $M=(m_{ij})$ such that  
\[F(x_1, x_2, \ldots, x_n)=G(y_1, y_2, \ldots, y_m), \]
if we make a substitution:     
\[\begin{pmatrix} x _1 &  x _2 & \cdots  & x _n  \end{pmatrix}
=
\begin{pmatrix} y _1 &  y _2 &  \cdots & y _m  \end{pmatrix}M.\]
\end{definition}

\begin{theorem}  \label{main_thm_2}  
Suppose that $A=\bigoplus _{i=0}^d A_i$ is a standard graded  Artinian  Gorenstein algebra over $A_0=K$, a field,  with socle degree $d$.  
Assume that the characteristic of $K$ is zero or greater than $d$.   
Suppose that $B=\bigoplus _{i=0}^d B_i$ is a Gorenstein subalgebra of $A=\bigoplus  _{i=0}^d A_i$ with the same socle degree. 
(We assume that $B=K[B_1]$, $A=K[A_1]$, and $B_1 \subset A_1$.)   
Then a Macaulay dual generator for $B$ 
is obtained from that of $A$ by substitution by linear forms. 
\end{theorem}

\begin{proof}  
We have shown that $\Phi=\Phi(\xi _1, \ldots, \xi _n)$ defined in Theorem 11 
is a Macaulay dual generator for $A$. 
Likewise we let 
$\Phi'=\Phi'(\eta _1, \ldots, \eta _m)$ be a Macaulay dual  generator for $B$. 
Let $s$ be a nonzero socle element of $A$. 
We may assume that 
$A_d=B_d=\langle s  \rangle$.  
Let $M=(m_{ij})$ be the $m\times n$ matrix which satisfies 
\[
\left(
\begin{array}{c}
 y_1 \\ y_2 \\ \vdots  \\ y_m 
\end{array}
\right)
=(m_{ij})
\left(
\begin{array}{c}
 x_1 \\ x_2 \\ \vdots  \\ x_n 
\end{array}
\right)
,\]
where $\langle x_1, x_2, \cdots, x_n   \rangle$ is a basis for $A_1$ and 
$\langle y_1, y_2, \cdots, y_m  \rangle$ for $B_1$.
Then, since 
\[
\begin{array}{rcl}
(\eta _1 y_1 + \eta _2 y_2  + \cdots + \eta _m y_m)^d 
&=& \left(\sum _{i=1}^m \eta _i (\sum _{j=1}^n m_{ij}x_j)\right)^d  \\[1ex] 
&=& \left(\sum _{j=1}^n(\sum _{i=1} ^m \eta _i m_{ij})x_j\right)^d,  \\[1ex] 
\end{array}
\]
we have 
\[
\Phi'(\eta_1,\ldots,\eta_m) 
= \Phi(\sum_{i=1}^m\eta_im_{i1}, \ldots, \sum_{i=1}^m\eta_im_{in}).
\]
This shows that $\Phi '$ is obtained from $\Phi$ by linear substitution of the 
variables with the matrix $(m_{ij})$: 
\[
(\xi _1 \ \xi _2  \ \cdots  \ \xi _n )
=(\eta _1 \ \eta _2 \ \cdots \ \eta _m) (m_{ij}). 
\]
\end{proof}

\begin{example}\normalfont 
Let  $R$ be the polynomial ring  in $n$ variables. 
Put $S_n$ be the symmetric group acting on $R$ by permutation of the variables. 
Let $I=(f_1, f_2, \ldots, f_n)$ be a quadratic complete intersection such that 
\[\Ann _R(F)=(f_1, \ldots, f_n).\] 
So $F$ is a Macaulay dual generator of $R/(f_1, \ldots , f_n)$.   
Suppose that $F^{\sigma }=F$ for all $\sigma \in S_n$. 
Let $\G$ be a Young subgroup of $S_n$ such that 
\[\G=S_{n_1}\times S_{n_2} \times \cdots \times S_{n_r},\]
\[n_1 + n_2 + \cdots + n_r=n.\]
Then the ring of invariants  $B=A^{\G} \subset A$ is a complete intersection and in many cases 
$A^{\G}$ is generated by linear forms (see \cite{harima_wachi_watanabe_1}).  
When  this is the case, the generators can be chosen as follows:  
\[y_1=x _1 + x _2 + \cdots + x _{n_1},\]
\[y_2={\underbrace{   x _{n_1+1} + x _{n_1+2} + \cdots + x _{n_1+n_2}   }_{n_2}   },\]
\[\vdots \]
\[y_r=\underbrace{    x _{   n_1 + \cdots + n_{r-1} + 1   } + x _{   n_1 + \cdots + n_{r-1} + 2   } + \cdots + x _n   }_{n_r}.\]

Then a Macaulay dual generator  $G$ for $B$ is obtained as follows:
\[G=F(\underbrace{y _1, \ldots, y _1}_{n_1}, 
\underbrace{y _2 ,\dots, y _2}_{n_2}, \cdots, \underbrace{y _r \ldots, y_r}_{n_r}).\]
\end{example}

\begin{example} 
Let $R=K[u,v,w,x,y,z]$ be the polynomial ring in 6 variables.
Put 
\[f_1=u^2-2u(v+w+x+y+z),\]
\[f_2=v^2-2v(u+w+x+y+z),\]
\[f_3=w^2-2w(u+v+x+y+z),\]
\[f_4=x^2-2x(u+v+w+y+z),\]
\[f_5=y^2-2y(u+v+w+x+z),\]
\[f_6=z^2-2z(u+v+w+x+y).\]
Let $A=R/(f_1, \cdots, f_6)$.
Then $A$ is an Artinian complete intersection.
A Macaulay dual generator is given as follows:
\begin{align*}
F=80m_6 &+48m_{51}+120m_{42}-30m_{411}+160m_{33}-60m_{321}+ \\ 
&\quad 60m_{3111}-90m_{222}+90m_{2211}-225m_{21111}+1575m_{111111}, 
\end{align*}
where $m_{ijk}$ etc.\ denotes the monomial symmetric polynomial. For
example,
\[m_6=u^6+v^6+w^6+x^6+y^6+z^6,\]
\[m_{51}=u^5v+uv^5+ \cdots +wz^5+xz^5+yz^5,\]
\[m_{411}=u^4vw + \cdots + xyz^4,\] 
\[\vdots\]
\[m_{111111}=uvwxuz.\]
The polynomial $F$ was obtained by solving a system of linear equations in 462
variables by Mathematica. (462 is the dimension of $K[u, \cdots, z]_6$.) 
The polynomial $F$  looks  complicated, but it has the  striking
property that any substitution of variables by another set of variables
defines a complete intersection. 
For example \[F(p,p,q,q,r,s) \in K[p,q,r,s]_6\] is a complete intersection in 4 variables.
This corresponds to the ring $A^{\G}$ of invariants by the Young subgroup 
\[\G:=S_{2}\times S_{2} \times S_{1} \times S_{1}.\]
\end{example}


\section{Subalgebras of Gorenstein algebras generated by linear forms} 

\begin{theorem} \label{second_main_thm}
Let $A=\bigoplus _{i=0}^d A_i$ be an Artinian Gorenstein algebra.  
Let $B$ be a subring of $A$ generated by a subspace of $A_1$ such that $B_d=A_d \neq 0$. 
Then 
we have the following:
\begin{enumerate}
\item[{\rm (1)}]  
There exists an irreducible ideal $\gb$ of  $B$ such that $B/\gb$ is a Gorenstein 
algebra with the socle degree $d$. 
\item[{\rm (2)}]
A Macaulay dual generator of $B'=B/\gb$ is obtained from that of $A$ by a substitution by 
linear forms. 
\end{enumerate}
\end{theorem}

\begin{proof}
\begin{enumerate}
\item[(1)]
Let 
\[\gb _1 \cap \gb _2 \cap \cdots \cap \gb _r =0\]  
be an irredundant  decomposition of $0$ in $B$ by irreducible ideals. Then we have the injection:
\[0 \to B \to \bigoplus _{i=1}^r B/\gb _i.\]
Note that there exists an $i$, say $i=1$,  such that $B'=B/\gb _1$ has the same 
socle as $B$.

\item[(2)] 
Let $\eta _1, \eta _2, \ldots, \eta _m$ be elements of $K$.  
The evaluation of the map $(\eta _1y_1 + \eta _2 y_2 + \cdots + \eta _m y_m)^d$ at $B_d$ or 
$B'_d$ are the same. Hence the assertion follows in the same way as Theorem~\ref{main_thm_2}.
\end{enumerate}
\end{proof}

\begin{example}  
Let $K$ be a field of characteristic zero. 
Consider $A=K[x,y,z]/I$, where 
\[I=(x^2-6yz,y^2-6xz, z^2-6xy).\]
Let $S=K[r,s]$, define $\psi:S \to A$  by $r \mapsto x, s\mapsto y+az$. 
Then we have 
\[\ker \psi =(3(a^3+1)rs^2-as^3, r^2s, (a^3+1)r^3-s^3),\]  
provided that  $a \neq \pm 1$. 

On the other hand a Macaulay dual generator for $A$ is $F=F(x,y,z)=x^3+y^3+z^3+xyz$.   
Let $G=F(r,s,as)$. That is,  the polynomial $G$ is obtained from $F$ by the substitution by the linear forms:
\[
(x,y,z)=(r,s)
\begin{pmatrix}
1&0&0 \\
0&1&a 
\end{pmatrix}
\]
A primary decomposition of the ideal $I$ is given by 
$I = \Ann _S(G)\cap (r+s, s^3)$, where 
\[\Ann _S(G)= (a^2r^2+9(a^3+1)rs - 3as^2, 3(a^3+1)rs^2-a s^3).\]
If $a=1$, then $\ker \psi$ is a complete intersection:
 \[\ker \psi=(r^2+18rs-3s^2, 6rs^2-s^3).\]
The case  $a=0$ works as well as other general cases.
The computation was done with the computer algebra system Macaulay2~\cite{grayson_stillman}.      
\end{example}

\appendix

\section{Appendix:Divided power algebra and the injective hull}

Let $R=K[x_1, x_2, \ldots, x_n]$ be the polynomial ring over a field $K$ of any characteristic.   
It is easy to see that 
\[\Hom _K(R,K)  = \prod _{i=0}^{\infty} \Hom _K(R_i, K)\]
is an injective $R$-module.  

In the category of finitely generated $R$-modules,  
we may adopt 
\[\Gamma= {\rm gr}.\Hom(R,K) = \bigoplus _{i=0}^{\infty} \Hom _K(R_i, K)\]
as the injective hull of the residue field of $K=R/R _{+}$. 
It is possible to endow the vector space $\Gamma$ a structure of a commutative 
algebra, called the divided power algebra. 
This can be explained as follows:
For the $K$ basis of $\Gamma$, we take the dual base of $R_i$, i.e, the ``monomials'' 
\[X_1 ^{(i_1)}X_2^{(i_2)}\cdots X_n^{(i_n)},\]
which are regarded as a homomorphism $R_{i_1 + \cdots + i_n} \to K$ defined by  
\[X_1^{(i_1)}X_2^{(i_2)}\cdots X_n^{(i_n)}\circ x_1^{j_1}x_2^{j_2} \cdots x_n^{j_n}=\left\{\begin{array}{l} 1, \; \mbox{if $I=J$,}\\ 0, \; \mbox{otherwise}. \end{array}\right. \]
($I$ and $J$ are multi-indices.) 
The multiplication among the monomials $\{X_1^{(i_1)}X_2^{(i_2)} \cdots X_n^{(i_n)}\}$
are defined by the rule:
\[X^IX^J={I + J \choose I}X^{I+J},\]
where
 \[X^I=X_1^{(i_1)} X_2^{(i_2)} \cdots X_n^{(i_n)},\]
\[X^J=X_1^{(j_1)} X_2^{(j_2)} \cdots X_n^{(j_n)},\]
\[I=(i_1, i_2, \ldots, i_n),\]  
\[J=(j_1, j_2, \ldots, j_n),\]
\[I+J=(i_1+j_1, i_2+j_2, \ldots, i_n+j_n),\]
\[{I+J \choose J}= \frac{(i_1+j_1)!(i_2+j_2)!\cdots (i_n+j_n)!}{i_1!i_2!\cdots i_n!j_1!j_2!\cdots j_n!}\]

Note that the coefficients ${I + J \choose J}$ are  integers. 
If the characteristic $p$ of $K$ is zero, we may think  $X^{(k)}=X^k/k!$, and 
the divided power algebra is the polynomial ring. If $p > 0$, then 
$p$th power is zero (which is easy to check), hence $\Gamma$ is not finitely generated.

We may regard $\Gamma$ as an $R$-module by the openration 
\[x_1^{i_1}x_2^{i_2}\cdots x_n^{i_n}\circ X_1 ^{(j_1)}X_2^{(j_2)} \cdots X_n^{(j_n)}=X_1 ^{(j_1-i_1)}X_2^{(j_2-i_2)} 
\cdots X_n^{(j_n-i_n)}.\]

At the start of Section 2, we set $\Gamma =R$ and the action of $f$ for  $F\in \Gamma$ be 
\[f\circ F=f\left(\frac{\pa}{\pa x_1}, \frac{\pa}{\pa x_2}, \ldots,  \frac{\pa}{\pa x_n}\right)F.\]
The reader may convince himself that this interpretation of the injective hull of 
$R/R_{+}$ and the construction of $\Gamma$ are consistent. 
For this it is enough to see that 
\[x_j\circ X_1^{(i_1)}X_2^{(i_2)} \cdots X_j^{(i_j)} \cdots X_n ^{(i_n)}=
X_1^{(i_1)}X_2^{(i_2)} \cdots X_j^{(i_j-1)} \cdots X_n ^{(i_n)}.\]
\noindent Thus the variable  $x_k$ acts on $\Gamma$ by ``differentiation.''

The reasons why we use this formulation are, among other things the following two propositions.
\begin{proposition}  
Let $F \in R_d$ and assume that $p=0$ or $p > d$.  
Then if the partials are dependent, then one variable can be eliminated from $F$ by 
a linear transformation of the variables. 
\end{proposition}

Proof is left to the reader.

\begin{proposition}     
Assume that $p=0$ or $p > d$.  
Let $F \in R_d$, and $A=A(F)$. 
Then the following conditions are equivalent. 
\begin{enumerate}
\item
The Hessian determinant of $F$ does not vanish. 
\item
There exists a linear form $l \in R_1$ such that 
the multiplication map 
\[l^{d-2}:A_2 \to A_{d-1}\]
is bijective. 
\end{enumerate}
\end{proposition}

For proof see \cite{maeno_watanabe}.  
The second proposition was not explicitly used in this paper, but it is a strong motivation for the usage of 
the divided power algebra.  

\vspace{1ex}
Suppose that $F \in \Gamma_d$.  Then $A(F)=R/\Ann _R(F)$ is an Artinian  Gorenstein algebra.   
Macaulay's double annihilator Theorem can be stated as follows: 

\begin{theorem}[F.\ S.\ Macaulay]
Let $R$ be the polynomial ring and $\Gamma$  the divided power algebra as the injective hull of $R/R_{+}$. 
  The correspondence $F \mapsto A(F)$ has the inverse. 
I.e., the set of graded Artinian Gorenstein algebra $A=R/I$ with top degree $d$ and   
the set of homogeneous forms in $\Gamma _d$ of degree $d$ are in one-to-one correspondence up to
linear change of variables over the field $K$.   
\end{theorem}

For proof see Meyer-Smith \cite[Theorem~II.2.1] {smith_1002}. See also the original work of Macaulay~\cite{macaulay}.

\begin{example}
  Let $R=K[x,y,z]$ and 
  consider  $F=(x+y+z)^3 \in K[x,y,z]_3$.
  We have defined $f\circ F =f(\frac{\pa}{\pa x},\frac{\pa}{\pa y},\frac{\pa}{\pa z})F$ for  $f,F \in R$. 
  Then we have 
  $\Ann _R(F)=(x-y, x-z, z^4)$. 
  F.\ S.\ Macaulay originally used   ``contract'' (rather than the differential operator) to make $R$ the injective hull for $R$.  
  In this case we have 
  \[\Ann _R(F)=(9x^2-10xz-4yz+15z^2,9y^2-10xy-4xz + 15x^2,  9z^2-10 yz -4 xy +15y^2). \]
In either case the double annihilator theorem holds. 
\end{example}


\begin{thebibliography}{99}

\bibitem{teitler}
W.\ Buczy\'{n}ska, J.\ Buczy\'{n}ski, J.\ Kleppe, and Z.\ Teitler, 
  \emph{Apolarity and direct sum decomposition of polynomials}, 
Michigan Math.\ J.\ Math., \textbf{64},(2015) no.\ 4, 675--719. 
\bibitem{GKZ}
I.\ M.\  Gelfand,  M.\  M.\ Kaplanov,  A.\ V.\ Zelevinski,  
Discriminants, Resultants, and Multidimensional Determinants,  
Birkh\"{a}user, Boston, 1994.   
\bibitem{gondim}
  R.\ Gondim,
\emph{On higher Hessians and the Lefschetz properties,}
To appear in J. Algebra. 
\bibitem{Ge} A.\ V.\  Geramita,
{\it Inverse systems of fat points: Waring's problem, secant varieties of Veronese varieties and parameter spaces for Gorenstein ideals},
 The Curves Seminar at Queen's, Vol. X, 2--114,
Queen's Papers in Pure and Appl. Math., 102, Queen's Univ., Kingston, ON, 1996.
\bibitem{grayson_stillman}
  D.\ R.\ Grayson and M.\ E.\ Stillman,
  Macaulay2, a software system for research in algebraic geometry, available at
  http://www.math.uiuc.edu/Macaulay2/.
\bibitem{IK}
A.\  Iarrobino, V.\  Kanev,
{\it Power sums, Gorenstein algebras, and determinantal loci},
 Lecture Notes in Mathematics, 1721, Springer-Verlag, Berlin, 1999.

\bibitem{smith_1002}   
  D.\ M.\ Meyer, L.\  Smith,
  Poincar\'{e} duality algebras, Macaulay's dual systems, and Steenrod operations,
  Cambridge Tracts in Mathematics, 167. Cambridge University Press, Cambridge, 2005. viii+193 pp.
  ISBN: 978-0-521-85064-3; 0-521-85064-9
  Cambridge University Press, Boston, 1994. 
\bibitem{harima_wachi_watanabe_1}
T.\ Harima, A.\ Wachi, J.\  Watanabe, 
\emph{The quadratic complete intersections associated with the action of the symmetric group,}
Illinois J.\ Math.\textbf{59} (1), 99--113 (2015), MR3459630
\bibitem{harima_wachi_watanabe_2} 
T.\ Harima, A.\ Wachi, J.\  Watanabe, 
{The resultants of binomial complete intersections,}
To appear.  
\bibitem{HMMNWW} 
T.\ Harima, T.\ Maeno, H.\ Morita, Y.\ Numata, A.\ Wachi,  and   J.\ Watanabe, 
The Lefschetz Properties, 
Springer Lecture Notes~2080, Springer-Verlag, 2013.
\bibitem{macaulay}
  F.\ S.\ Macaulay.
  The Algebraic Theory of Modular Systems,
  Camb.\ Math.\ Lib.\ , Cambridge: Cambridge University Press, 1996
  (reissued with an introduction by P.\ Roberts 1994).

\bibitem{maeno_watanabe}
T.\ Maeno, J.\  Watanabe, 
\emph{Lefschetz elements of Artinian Gorenstein algebras and Hessians of homogeneous polynomials,}
Illinois J.\ Math.\textbf{53} (2), 591--603 (2009). 
\bibitem{solomon_1}
{L.\ Solomon}, 
{\it Invariants of Finite Reflection Groups}, 
Nagoya Math.\ J.\ 22 (1963)  57--64. 
\bibitem{mcdaniel} 
C.\ McDaniel, 
\emph{The strong Lefschetz properties for coinvariant rings of finite 
reflection groups}, 
J.\ Algebra\  \textbf{331}, (2003),  99--126.  
\bibitem{watanabe_1}
  J.\ Watanabe,
  \emph{A remark on the Hessian of homogeneous polynomials}, 
The curves seminar at Queen's, vol. XIII, Queen's Papers in Pure and Appl.\ Math., vol. 119, 2000, Queen's Univ.\ ON, 171--178. 
\end{thebibliography}
\end{document}